\documentclass[12pt]{article}
\usepackage{amsmath,amsthm,amssymb}
\usepackage{amssymb,latexsym}

\newtheorem{theorem}{Theorem}

\newtheorem{lemma}{Lemma}

\begin{document}
\baselineskip=17pt

\title{\bf On the distribution of {\boldmath$\alpha p^2$} modulo one over primes of the form {\boldmath$[n^c]$}}

\author{\bf S. I. Dimitrov \and \bf M. D. Lazarova}

\date{2025}

\maketitle

\begin{abstract}
Let $[\, \cdot\,]$ be the floor function and $\|x\|$ denote the distance from $x$ to the nearest integer.
In this paper we show that whenever $\alpha$ is irrational and $\beta$ is real then for any fixed $\frac{13}{14}<\gamma<1$, there exist infinitely many prime numbers $p$ satisfying the inequality
\begin{equation*}
\|\alpha p^2+\beta\|< p^{\frac{13-14\gamma}{29}+\varepsilon}
\end{equation*}
and such that $p=[n^{1/\gamma}]$.\\
\quad\\
\textbf{Keywords}:  Distribution modulo one, Piatetski-Shapiro primes.\\
\quad\\
{\bf  2020 Math.\ Subject Classification}: 11J25 $\cdot$ 11J71 $\cdot$ 11L07 $\cdot$ 11L20
\end{abstract}

\section{Introduction and statement of the result}
\indent

The existence of infinitely many prime numbers of a special form is one of the biggest challenge in prime number theory.
There are not many thin sets of primes about which we have the asymptotic formula for their distribution.
In 1953, Piatetski-Shapiro \cite{Shapiro} showed that for any fixed $\frac{11}{12}<\gamma<1$, there exist infinitely many prime numbers of the form $p = [n^{1/\gamma}]$.
Such primes are called Piatetski-Shapiro primes of type $\gamma$.
Subsequently the interval for $\gamma$ was improved by many authors and the best result to date has been supplied by Rivat and Wu \cite{Rivat-Wu}.
More precisely they showed that for any fixed $\frac{205}{243}<\gamma<1$ we have
\begin{equation}\label{Louerbound}
\sum\limits_{p\leq X\atop{p=[n^{1/\gamma}]}}1\gg\frac{X^\gamma}{\log X}\,.
\end{equation}
On the other hand in 1947 Vinogradov \cite{Vinogradov} proved that if $\theta=\frac{1}{5}-\varepsilon$, then there are infinitely many primes $p$ such that
\begin{equation}\label{Vinogradov-ineq}
\|\alpha p+\beta\|<p^{-\theta }\,.
\end{equation}
Afterwards, inequality \eqref{Vinogradov-ineq} was sharpened several times and the best result up to now belongs to Matom\"{a}ki \cite{Mato} with $\theta=\frac{1}{3}-\varepsilon$ and $\beta=0$.

Recently, Dimitrov \cite{Dimitrov} considered a hybrid problem, restricting the set of primes $p$ in \eqref{Vinogradov-ineq} to Piatetski-Shapiro primes.
To be specific, he proved that, for any fixed $\frac{11}{12}<\gamma<1$, there exist infinitely many Piatetski-Shapiro primes $p$ of type $\gamma$ such that
\begin{equation*}
\|\alpha p+\beta\|< p^{\frac{11-12\gamma}{26}+\varepsilon}\,.
\end{equation*}
In turn X. Li, J. Li and Zhang \cite{Li} generalized the result of Dimitrov \cite{Dimitrov} by solving \eqref{Vinogradov-ineq} with primes $p=[n_1^{1/\gamma_1}]=[n_2^{1/\gamma_2}]$, 
where $\frac{23}{12}<\gamma_1+\gamma_2<2$ and with $\theta=\frac{12(\gamma_1+\gamma_2)-23}{38}-\varepsilon$.
Very recently Baier and Rahaman \cite{Baier} managed to improve Dimitrov's result by solving \eqref{Vinogradov-ineq} with primes $p=[n^{1/\gamma}]$, where $\frac{8}{9}<\gamma<1$
and with $\theta=\frac{9\gamma-8}{10}-\varepsilon$.

The researchers solved inequality \eqref{Vinogradov-ineq} with higher powers of $p$. Ghosh \cite{Ghosh}  is credited with the inequality
\begin{equation}\label{Ghosh-ineq}
\|\alpha p^2+\beta\|<p^{-\theta}\,,
\end{equation}
which is valid for infinitely many primes $p$ and $\theta=\frac{1}{8}-\varepsilon$.
Subsequently, the result of Ghosh  was sharpened by Baker and Harman \cite{Baker} with $\theta=\frac{3}{20}-\varepsilon$ and by Harman \cite{Harman1996} with $\theta=\frac{2}{13}-\varepsilon$.
As a continuation of these studies, we solve inequality \eqref{Ghosh-ineq} with Piatetski-Shapiro primes. 
\begin{theorem}\label{Theorem}
Let $\gamma$ be fixed with $\frac{13}{14}<\gamma<1$, $\alpha$ is irrational and $\beta$ is real.
Then there exist infinitely many Piatetski-Shapiro primes $p$ of type $\gamma$ such that
\begin{equation*}
\|\alpha p^2+\beta\|<p^{\frac{13-14\gamma}{29}+\varepsilon}\,.
\end{equation*}
\end{theorem}

We remark that Theorem \ref{Theorem} is unlikely to be best possible. It is plausible that more refined exponential sum estimates and/or sieve methods could further extend the admissible range of $\gamma$. 
However, we have chosen not to pursue such refinements here, as our primary aim is to demonstrate that inequality \eqref{Ghosh-ineq} admits infinitely many solutions in Piatetski-Shapiro primes.

\section{Notations}
\indent

Let $C$ be a sufficiently large positive constant. The letter $p$ will always denote a prime number.
By $\varepsilon$ we denote an  arbitrarily small positive number, not the same in all appearances.
The notation $m\sim M$ means that $m$ runs through the interval $(M, 2M]$.
As usual $\Lambda(n)$ is von Mangoldt's function and $\tau(n)$ denotes the number of positive divisors of $n$.
By $[x]$, $\{x\}$ and $\|x\|$ we denote the integer part of $x$, the fractional part of $x$ and the distance from $x$ to the nearest integer.
Moreover $e(x)=e^{2\pi i x}$ and $\psi(t)=\{t\}-1/2$.
Let $\gamma$ be a real constant such that $\frac{13}{14}<\gamma<1$.
Since $\alpha$ is irrational, there are infinitely many different convergents $a/q$ to its continued fraction, with
\begin{equation}\label{alphaaq}
\bigg|\alpha- \frac{a}{q}\bigg|<\frac{1}{q^2}\,,\quad (a, q) = 1\,,\quad a\neq0
\end{equation}
and $q$ is arbitrary large. Denote
\begin{align}
\label{Nq}
&N=q^\frac{29}{55-28\gamma}\,;\\
\label{Delta}
&\Delta=CN^{\frac{13-14\gamma}{29}+\varepsilon}\,;\\
\label{HNq}
&H=\big[q^\frac{1}{2}\big]\,;\\
\label{Mgamma}
&M=N^{\frac{16-15\gamma}{29}}\,;\\
\label{vgamma}
&\vartheta=N^{\frac{2\gamma+23}{58}}\,;\\
\label{Sigma}
&\Sigma=\sum\limits_{p\leq N}\Big(\psi(-(p+1)^\gamma)-\psi(-p^\gamma)\Big)e(\alpha h p^2)\log p\,.
\end{align}

\section{Preliminary lemmas}
\indent

\begin{lemma}\label{Summin2}
Suppose that $H, N\geq1$,\, $\big|\alpha-\frac{a}{q}\big|<\frac{1}{q^2}$\,, $(a, q)=1$.
Then
\begin{equation*}
\sum_{n\le N} \, \min \left(1,\, \frac{1}{H \|\alpha n^2+\beta\pm\Delta\| } \right) \ll (NHq)^\varepsilon\Big(Nq^{-\frac{1}{2}}+N^\frac{1}{2}+NH^{-1}+H^{-\frac{1}{2}}q^\frac{1}{2}\Big)\,.
\end{equation*}
\end{lemma}
\begin{proof}
See (\cite{Ghosh}, pp. 265 -- 266).
\end{proof}

\begin{lemma}\label{Expsumest}
Suppose that $\alpha \in \mathbb{R}$,\, $a \in \mathbb{Z}$,\, $q\in \mathbb{N}$,\, $\big|\alpha-\frac{a}{q}\big|\leq\frac{1}{q^2}$\,, $(a, q)=1$.
Then
\begin{equation*}
\sum\limits_{p\le N}e(\alpha p^2)\log p\ll N^{1+\varepsilon}\bigg(\frac{1}{q}+\frac{1}{N^\frac{1}{2}}+\frac{q}{N^2}\bigg)^\frac{1}{4}\,.
\end{equation*}
\end{lemma}
\begin{proof}
See (\cite{Ghosh}, Theorem 2).
\end{proof}

\begin{lemma}\label{Expansion}
For any $M\geq2$, we have
\begin{equation*}
\psi(t)=-\sum\limits_{1\leq|m|\leq M}\frac{e(mt)}{2\pi i m}+\mathcal{O}\Bigg(\min\left(1, \frac{1}{M\|t\|}\right)\Bigg)\,,
\end{equation*}
\begin{proof}
See (\cite{Tolev1}, Lemma 5.2.2).
\end{proof}
\end{lemma}

\begin{lemma}\label{Sargos}
Suppose that $f'''(t)$ exists, is continuous on $[a,b]$ and satisfies
\begin{equation*}
f'''(t)\asymp\lambda\quad(\lambda>0)\quad\mbox{for}\quad t\in[a,b]\,.
\end{equation*}
Then
\begin{equation*}
\bigg|\sum_{a<n\le b}e(f(n))\bigg|\ll(b-a)\lambda^\frac{1}{6}+\lambda^{-\frac{1}{3}}\,.
\end{equation*}
\end{lemma}
\begin{proof}
See (\cite{Sargos}, Corollary 4.2).
\end{proof}

\begin{lemma}\label{Iwaniec-Kowalski}
For any complex numbers $a(n)$ we have
\begin{equation*}
\bigg|\sum_{a<n\le b}a(n)\bigg|^2\leq\bigg(1+\frac{b-a}{Q}\bigg)\sum_{|q|\leq Q}\bigg(1-\frac{|q|}{Q}\bigg)\sum_{a<n,\, n+q\leq b}a(n+q)\overline{a(n)}\,,
\end{equation*}
where $Q$ is any positive integer.
\end{lemma}
\begin{proof}
See (\cite{Iwaniec-Kowalski}, Lemma 8.17).
\end{proof}

\section{Proof of the theorem}

\subsection{Beginning of the proof}

\indent

Our method goes back to Vaughan \cite{Vaughan1}. We take a periodic with period 1 function such that
\begin{equation*}
F_\Delta(\theta)=\begin{cases}
0\quad \mbox{if} \quad -\frac{1}{2}\leq \theta< -\Delta\,,\\
1\quad\mbox{if}\quad-\Delta\leq \theta< \Delta\,,\\
0\quad\mbox{if}\quad\; \Delta\leq \theta<\frac{1}{2}\,,
\end{cases}
\end{equation*}
where $\Delta$ is defined by \eqref{Delta}. Any non-trivial estimate from below of the sum
\begin{equation*}
\sum\limits_{p\leq N\atop{p=[n^{1/\gamma}]}}F_\Delta(\alpha p^2+\beta)\log p
\end{equation*}
implies Theorem \ref{Theorem}.
For this goal we define
\begin{equation}\label{Gamma}
\Gamma=\sum\limits_{p\leq N\atop{p=[n^{1/\gamma}]}}\big(F_\Delta(\alpha p^2+\beta)-2\Delta\big)\log p\,.
\end{equation}

\subsection{Estimation of $\mathbf{\Gamma}$}
\indent

\begin{lemma}\label{Sigmaest} Let $\frac{13}{14}<\gamma<1$. For the sum denoted by \eqref{Sigma} the upper bound
\begin{equation*}
\Sigma\ll N^{\frac{15\gamma+13}{29}+\varepsilon}
\end{equation*}
holds.
\end{lemma}
\begin{proof}
By \eqref{Mgamma}, \eqref{Sigma}, Lemma \ref{Expansion} and the simplest splitting up argument, we write
\begin{equation}\label{Sigmaest1}
\Sigma\ll\big(\Sigma_1+\Sigma_2\big)\log^2N+N^{1/2}\,,
\end{equation}
where
\begin{align}
\label{Sigma1}
&\Sigma_1=\sum\limits_{m\sim M_1}\frac{1}{m}\left|\sum\limits_{n\sim N_1}\Lambda(n)e(\alpha h n^2)\Big(e\big(-mn^\gamma\big)-e\big(-m(n+1)^\gamma\big)\Big) \right|\,,\\
\label{Sigma2}
&\Sigma_2=\sum\limits_{n\sim N_1}\min\left(1, \frac{1}{M\|n^\gamma\|}\right)\,,\\
\label{M1N1}
&M_1\leq \frac{M}{2}\,,\quad N_1\leq \frac{N}{2}\,.
\end{align}
Arguing as in (\cite{Tolev1}, Theorem 12.1.1), and using \eqref{Sigma2} and \eqref{M1N1}, we obtain
\begin{equation}\label{Sigma2est1}
\Sigma_2\ll \Big(NM^{-1}+N^\frac{\gamma}{2} M^\frac{1}{2}+N^{1-\frac{\gamma}{2}}M^{-\frac{1}{2}}\Big)\log M\,.
\end{equation}
Taking into account \eqref{Mgamma} and \eqref{Sigma2est1}, we get
\begin{equation}\label{Sigma2est2}
\Sigma_2\ll N^{\frac{15\gamma+13}{29}+\varepsilon}\,.
\end{equation}
Next we estimate $\Sigma_1$. Put
\begin{equation*}
\lambda(t)=1-e\big(m(t^\gamma-(t+1)^\gamma)\big)\,.
\end{equation*}
Applying Abel's summation formula, we derive
\begin{align}\label{Abelsum4}
&\sum\limits_{n\sim N_1}\Lambda(n)e(\alpha h n^2)\Big(e\big(-mn^\gamma\big)-e\big(-m(n+1)^\gamma\big)\Big)\nonumber\\
&=\lambda(2N_1)\sum\limits_{n\sim N_1}\Lambda(n)e\big(\alpha hn^2-mn^\gamma\big)-\int\limits_{N_1}^{2N_1}\left(\sum\limits_{N_1<n\leq t}\Lambda(n)e\big(\alpha hn^2-mn^\gamma\big)\right)\lambda'(t)\,dt\nonumber\\
&\ll m N^{\gamma-1}_1\max_{N_2\in[N_1,2N_1]}|\Phi(N_1,N_2)|\,,
\end{align}
where
\begin{equation}\label{Phi}
\Phi(N_1,N_2)=\sum\limits_{N_1<n\leq N_2}\Lambda(n)e\big(\alpha hn^2-mn^\gamma\big)\,.
\end{equation}
Now \eqref{Sigma1} and \eqref{Abelsum4} imply
\begin{equation}\label{Sigma1est1}
\Sigma_1\ll N^{\gamma-1}_1\sum\limits_{m\sim M_1}\max_{N_2\in[N_1,2N_1]}|\Phi(N_1,N_2)|\,.
\end{equation}
Suppose that
\begin{equation}\label{N1less}
N_1\leq N^{\frac{30\gamma-3}{29\gamma}}\,.
\end{equation}
Bearing in mind \eqref{Mgamma}, \eqref{M1N1}, \eqref{Phi}, \eqref{Sigma1est1}  and \eqref{N1less}, we deduce
\begin{equation}\label{Sigma1est2}
\Sigma_1\ll N^{\frac{15\gamma+13}{29}}\,.
\end{equation}
From now on we assume that
\begin{equation}\label{N1greater}
N^{\frac{30\gamma-3}{29\gamma}}<N_1\leq 2N\,.
\end{equation}
We shall estimate the sum \eqref{Phi}. Put
\begin{equation}\label{fdl}
f(d,l)=\alpha hd^2l^2-md^\gamma l^\gamma\,.
\end{equation}
Using \eqref{Phi}, \eqref{fdl} and Vaughan's identity (see \cite{Vaughan2}), we write
\begin{equation}\label{Thetadecomp}
\Phi(N_1,N_2)=\Theta_1-\Theta_2-\Theta_3-\Theta_4\,,
\end{equation}
where
\begin{align}
\label{Theta1}
&\Theta_1=\sum_{d\le \vartheta}\mu(d)\sum_{\frac{N_1}{d}<l\le \frac{N_2}{d}}e(f(d,l))\log l\,,\\
\label{Theta2}
&\Theta_2=\sum_{d\le \vartheta}c(d)\sum_{\frac{N_1}{d}<l\le \frac{N_2}{d}}e(f(d,l))\,,\\
\label{Theta3}
&\Theta_3=\sum_{\vartheta<d\le \vartheta^2}c(d)\sum_{\frac{N_1}{d}<l\le \frac{N_2}{d}}e(f(d,l))\,,\\
\label{Theta4}
&\Theta_4= \mathop{\sum\sum}_{\substack{N_1<dl\le N_2 \\d>\vartheta,\,l>\vartheta}}a(d)\Lambda(l) e(f(d,l))
\end{align}
and 
\begin{equation}\label{cdad}
|c(d)|\leq\log d,\quad  | a(d)|\leq\tau(d)\,,
\end{equation}
and $\vartheta$ is defined by \eqref{vgamma}. Consider first the sum $\Theta_2$ defined by \eqref{Theta2}. Taking into account \eqref{fdl}, we obtain
\begin{equation}\label{f''ll}
|f^{'''}_{lll}(d,l)|\asymp m d^3N_1^{\gamma-3}\,.
\end{equation}
Now \eqref{f''ll} and Lemma \ref{Sargos} yield
\begin{equation}\label{sumefdl}
\sum_{\frac{N_1}{d}<l\le \frac{N_2}{d}}e(f(d, l))\ll m^\frac{1}{6}d^{-\frac{1}{2}} N_1^{\frac{\gamma}{6}+\frac{1}{2}}+m^{-\frac{1}{3}}d^{-1}N_1^{1-\frac{\gamma}{3}}\,.
\end{equation}
From  \eqref{Mgamma}, \eqref{vgamma}, \eqref{Theta2}, \eqref{cdad} and \eqref{sumefdl}, we get
\begin{equation}\label{Theta2est1}
\Theta_2\ll\Big(m^\frac{1}{6}\vartheta^\frac{1}{2} N_1^{\frac{\gamma}{6}+\frac{1}{2}}+m^{-\frac{1}{3}}N_1^{1-\frac{\gamma}{3}}\Big)N^\varepsilon
\ll m^\frac{1}{6}\vartheta^\frac{1}{2} N_1^{\frac{\gamma}{6}+\frac{1}{2}}N^\varepsilon\,.
\end{equation}
In order to estimate $\Theta_1$ defined by \eqref{Theta1}, we apply Abel's summation formula. Then, applying the same method as for $\Theta_2$, we find
\begin{equation}\label{Theta1est1}
\Theta_1\ll m^\frac{1}{6}\vartheta^\frac{1}{2} N_1^{\frac{\gamma}{6}+\frac{1}{2}}N^\varepsilon\,.
\end{equation}
It remains to estimate the sums $\Theta_3$ and $\Theta_4$. By \eqref{Theta4}, we have 
\begin{equation}\label{Theta44'}
\Theta_4\ll|\Theta'_4|\log N_1\,,
\end{equation}
where
\begin{equation}\label{Theta'4}
\Theta'_4=\sum_{D<d\le 2D}a(d)\sum_{L<l\le 2L\atop{N_1<dl\le N_2}}\Lambda(l)e(f(d,l))
\end{equation}
and where
\begin{equation*}
\frac{N_1}{4}\leq DL\le 2N_1\,, \quad  \frac{\vartheta}{2}\leq D\leq  \frac{2N_1}{\vartheta}\,.
\end{equation*}
Arguing as in  \cite{Dimitrov} we conclude that it is sufficient to estimate the sum $\Theta'_4$ with the conditions
\begin{equation}\label{ParTheta'4}
\frac{N_1}{4}\leq DL\le 2N_1\,, \quad  \frac{N^\frac{1}{2}_1}{2}\leq D\leq \vartheta^2\,.
\end{equation}
Then the obtained estimate for $\Theta_4$ will be valid for $\Theta_3$.
Using \eqref{cdad}, \eqref{Theta'4}, \eqref{ParTheta'4}, Cauchy's inequality and Lemma \ref{Iwaniec-Kowalski} with $Q\leq \frac{L}{2}$, we derive
\begin{equation}\label{Theta'4est1}
|\Theta'_4|^2\ll\Bigg(\frac{LD}{Q}\sum_{1\leq q\leq Q}\sum_{L<l\le 2L}\bigg|\sum_{D_1<d\le D_2}e(g(d))\bigg|+\frac{ (LD)^2}{Q}\Bigg)N^\varepsilon\,,
\end{equation}
where
\begin{equation}\label{maxmin2}
D_1=\max{\bigg\{D,\frac{N_1}{l},\frac{N_1}{l+q}\bigg\}},\quad
D_2=\min{\bigg\{2D,\frac{N_2}{l},\frac{N_2}{l+q}\bigg\}}
\end{equation}
and
\begin{equation}\label{gd}
g(d)=f(d,l+q)-f(d,l)\,.
\end{equation}
Consider the function $g(d)$. From  \eqref{fdl} and \eqref{gd}, we deduce
\begin{equation}\label{g''d}
|g'''(d)|\asymp m D^{\gamma-3} q L^{\gamma-1}\,.
\end{equation}
Now \eqref{maxmin2}, \eqref{g''d} and Lemma \ref{Sargos} give us
\begin{equation}\label{sumegd}
\sum\limits_{D_1<d\leq D_2}e(g(d))
\ll m^\frac{1}{6}q^\frac{1}{6}D^{\frac{\gamma}{6}+\frac{1}{2}}L^{\frac{\gamma}{6}-\frac{1}{6}}+m^{-\frac{1}{3}}q^{-\frac{1}{3}}D^{1-\frac{\gamma}{3}}L^{\frac{1}{3}-\frac{\gamma}{3}}\,.
\end{equation}
We choose
\begin{equation}\label{Qmin}
Q=\min\Big([L/4]\,, [Q_0]\Big)\,,
\end{equation}
where
\begin{equation}\label{Q0mDL}
Q_0=m^{-\frac{1}{7}}D^\frac{3-\gamma}{7}L^\frac{1-\gamma}{7}\,.
\end{equation}
By \eqref{Mgamma}, \eqref{M1N1}, \eqref{N1greater}, \eqref{ParTheta'4} and \eqref{Q0mDL}, it follows that 
\begin{equation*}
Q_0>N^{\frac{332}{2639}}\,.
\end{equation*}
Taking into account \eqref{Theta'4est1}, \eqref{sumegd}, \eqref{Qmin} and \eqref{Q0mDL}, we obtain
\begin{align}\label{Theta'4est2}
|\Theta'_4|^2&\ll\big( D^2L^2Q^{-1}+m^\frac{1}{6}Q^\frac{1}{6}D^{\frac{\gamma}{6}+\frac{3}{2}}L^{\frac{\gamma}{6}+\frac{11}{6}}
+m^{-\frac{1}{3}}Q^{-\frac{1}{3}}D^{2-\frac{\gamma}{3}}L^{\frac{7}{3}-\frac{\gamma}{3}}\big)N^\varepsilon\nonumber\\
&\ll \Big(D^2L^2L^{-1}+ D^2L^2Q_0^{-1}+m^\frac{1}{6}Q_0^\frac{1}{6}D^{\frac{\gamma}{6}+\frac{3}{2}}L^{\frac{\gamma}{6}+\frac{11}{6}}\nonumber\\
&\hspace{47mm}+m^{-\frac{1}{3}}D^{2-\frac{\gamma}{3}}L^{\frac{7}{3}-\frac{\gamma}{3}}\big(L^{-\frac{1}{3}}+Q_0^{-\frac{1}{3}}\big)\Big)N^\varepsilon\nonumber\\
&\ll \Big( D^2L+m^\frac{1}{7}D^{\frac{\gamma}{7}+\frac{11}{7}}L^{\frac{\gamma}{7}+\frac{13}{7}}+m^{-\frac{1}{3}}D^{2-\frac{\gamma}{3}}L^{2-\frac{\gamma}{3}}
+m^{-\frac{2}{7}}D^{\frac{13}{7}-\frac{2\gamma}{7}}L^{\frac{16}{7}-\frac{2\gamma}{7}}\Big)N^\varepsilon\,.
\end{align}
Now \eqref{Theta44'}, \eqref{ParTheta'4} and \eqref{Theta'4est2} lead to
\begin{equation}\label{Theta4est1}
\Theta_4\ll \Big( N_1^\frac{1}{2}\vartheta+M^\frac{1}{14}N_1^{\frac{\gamma}{14}+\frac{6}{7}}\Big)N^\varepsilon\,.
\end{equation}
Working as in the estimation of $\Theta_4$ for the sum \eqref{Theta3}, we get
\begin{equation}\label{Theta3est1}
\Theta_3\ll \Big( N_1^\frac{1}{2}\vartheta+M^\frac{1}{14}N_1^{\frac{\gamma}{14}+\frac{6}{7}}\Big)N^\varepsilon\,.
\end{equation}
Summarizing \eqref{Thetadecomp}, \eqref{Theta2est1}, \eqref{Theta1est1}, \eqref{Theta4est1} and \eqref{Theta3est1}, we derive
\begin{equation}\label{Thetaest}
\Theta(N_1,N_2)\ll \Big( N_1^\frac{1}{2}\vartheta+M^\frac{1}{14}N_1^{\frac{\gamma}{14}+\frac{6}{7}}+m^\frac{1}{6}\vartheta^\frac{1}{2} N_1^{\frac{\gamma}{6}+\frac{1}{2}}\Big)N^\varepsilon\,.
\end{equation}
By  \eqref{Mgamma}, \eqref{vgamma}, \eqref{Sigma1est1}, \eqref{N1greater} and \eqref{Thetaest}, it follows that
\begin{equation}\label{Sigma1est3}
\Sigma_1\ll N^{\frac{15\gamma+13}{29}+\varepsilon}\,.
\end{equation}
Bearing in mind \eqref{Sigmaest1}, \eqref{Sigma2est2}, \eqref{Sigma1est2} and \eqref{Sigma1est3}, we establish the statement in the lemma.

\end{proof}

\begin{lemma}\label{Mainlemma}
Let $\frac{13}{14}<\gamma<1$. For the sum $\Gamma$ defined by \eqref{Gamma} the estimate
\begin{equation*}
\Gamma\ll N^{\frac{15\gamma+13}{29}+\varepsilon}
\end{equation*}
holds.
\end{lemma}
\begin{proof}

From \eqref{Gamma}, we have
\begin{equation}\label{Gammadecomp}
\Gamma=\sum\limits_{p\leq N}\big([-p^\gamma]-[-(p+1)^\gamma]\big)\big(F_\Delta(\alpha p^2+\beta)-2\Delta\big)\log p
=\Gamma_1+\Gamma_2\,,
\end{equation}
where
\begin{align}
\label{Gamma1}
&\Gamma_1=\sum\limits_{p\leq N}\big((p+1)^\gamma-p^\gamma\big)\big(F_\Delta(\alpha p^2+\beta)-2\Delta\big)\log p\,,\\
\label{Gamma2}
&\Gamma_2=\sum\limits_{p\leq N}\big(\psi(-(p+1)^\gamma)-\psi(-p^\gamma)\big)\big(F_\Delta(\alpha p^2+\beta)-2\Delta\big)\log p\,.
\end{align}

\textbf{Upper bound for} $\mathbf{\Gamma_1}$
\indent

The function $F_\Delta(\theta)-2\Delta$ is well known to have the expansion
\begin{equation}\label{fDeltaexpansion}
\sum\limits_{1\leq|h|\leq H}\frac{\sin2\pi h\Delta}{\pi h}\,e(h\theta)+\mathcal{O}\Bigg(\min\left(1, \frac{1}{H\|\theta+\Delta\|}\right)+\min\left(1, \frac{1}{H\|\theta-\Delta\|}\right)\Bigg)\,.
\end{equation}
We also have
\begin{equation}\label{p+1}
(p+1)^\gamma-p^\gamma=\gamma p^{\gamma-1}+\mathcal{O}\left(p^{\gamma-2}\right)\,.
\end{equation}
Now \eqref{Gamma1}, \eqref{fDeltaexpansion} and \eqref{p+1}, give us
\begin{equation}\label{Gamma1est1}
\Gamma_1=\gamma\sum\limits_{p\leq N}p^{\gamma-1}\log p
\sum\limits_{1\leq|h|\leq H}\frac{\sin2\pi h\Delta}{\pi h}\,e\big(h(\alpha p^2+\beta)\big)+\mathcal{O}\big(\Omega\log N\big)\,,
\end{equation}
where
\begin{equation}\label{Omega}
\Omega=\sum\limits_{n=1}^N\Bigg(\min\left(1, \frac{1}{H\|\alpha n^2+\beta+\Delta\|}\right)+\min\left(1, \frac{1}{H\|\alpha n^2+\beta-\Delta\|}\right)\Bigg)\,.
\end{equation}
From \eqref{alphaaq}, \eqref{Nq}, \eqref{HNq}, \eqref{Omega} and Lemma \ref{Summin2}, we obtain
\begin{equation}\label{Omegaest}
\Omega\ll N^\varepsilon\Big(Nq^{-\frac{1}{2}}+N^\frac{1}{2}+NH^{-1}+H^{-\frac{1}{2}}q^\frac{1}{2}\Big)\ll N^{1+\varepsilon} q^{-\frac{1}{2}}\ll N^{\frac{28\gamma+3}{58}+\varepsilon}\,.
\end{equation}
Now \eqref{Gamma1est1} and \eqref{Omegaest} imply
\begin{equation}\label{Gamma1est2}
\Gamma_1\ll\sum\limits_{h=1}^{H}\min\left(\Delta,\frac{1}{h}\right)\left|\sum\limits_{p\leq N}p^{\gamma-1}e(\alpha h p^2)\log p \right|+N^{\frac{28\gamma+3}{58}+\varepsilon}\,.
\end{equation}
Put
\begin{equation}\label{mathfrakS}
\mathfrak{S}(u)=\sum\limits_{h\leq u}\left|\sum\limits_{p\leq N}p^{\gamma-1}e(\alpha h p^2)\log p \right|\,.
\end{equation}
Using Abel's summation formula, we get
\begin{align}\label{Abelsum1}
\sum\limits_{h=1}^{H}\min\left(\Delta,\frac{1}{h}\right)\left|\sum\limits_{p\leq N}p^{\gamma-1}e(\alpha h p^2)\log p \right|
&=\frac{\mathfrak{S}(H)}{H}+\int\limits_{\Delta^{-1}}^{H}\frac{\mathfrak{S}(u)}{u^2}\,du\nonumber\\
&\ll(\log H)\max_{\Delta^{-1}\leq u\leq H}\frac{\mathfrak{S}(u)}{u}\,.
\end{align}
Applying Abel's summation formula again, we deduce
\begin{equation}\label{Abelsum2}
\sum\limits_{p\leq N}p^{\gamma-1}e(\alpha h p^2)\log p=N^{\gamma-1}S(N)+(1-\gamma)\int\limits_{2}^{N}S(y)y^{\gamma-2}\,dy\,,
\end{equation}
where
\begin{equation}\label{Sy}
S(y)=\sum\limits_{p\leq y}e(\alpha h p^2)\log p\,.
\end{equation}
From Dirichlet's approximation theorem it follows the existence of integers $a_h$ and $q_h$ such that
\begin{equation}\label{alphahahqh}
\bigg|\alpha h-\frac{a_h}{q_h}\bigg|\leq\frac{1}{q_hq^2}\,,
\quad (a_h,q_h)=1\,,\quad 1\leq q_h\leq q^2\,.
\end{equation}
Taking into account \eqref{Sy}, \eqref{alphahahqh} and Lemma \ref{Expsumest}, we derive
\begin{equation}\label{Syest}
S(y)\ll y^{1+\varepsilon}\Big(q_h^{-\frac{1}{4}}+y^{-\frac{1}{8}}+y^{-\frac{1}{2}}q_h^{\frac{1}{4}}\Big)\,.
\end{equation}
By \eqref{mathfrakS}, \eqref{Abelsum2} and \eqref{Syest}, we obtain
\begin{equation}\label{mathfrakSest1}
\mathfrak{S}(u)\ll N^{\gamma-1+\varepsilon}\sum\limits_{h\leq u}\Big(Nq_h^{-\frac{1}{4}}+N^{\frac{7}{8}}+N^{\frac{1}{2}}q_h^{\frac{1}{4}}\Big)\,.
\end{equation}
Using \eqref{alphaaq}, \eqref{HNq}, \eqref{alphahahqh} and arguing as in \cite{Dimitrov}, we conclude that
\begin{equation}\label{qhlimits}
q_h\in\Big(q^\frac{1}{3},q^2\Big]\,.
\end{equation}
Bearing in mind \eqref{Nq}, \eqref{mathfrakSest1} and \eqref{qhlimits}, we get
\begin{equation}\label{mathfrakSest2}
\mathfrak{S}(u)\ll uN^{\gamma-\frac{1}{2}+\varepsilon}q^\frac{1}{2}\ll uN^{\frac{15\gamma+13}{29}+\varepsilon}\,.
\end{equation}
Summarizing \eqref{Nq}, \eqref{HNq}, \eqref{Gamma1est2}, \eqref{Abelsum1} and \eqref{mathfrakSest2}, we obtain 
\begin{equation}\label{Gamma1est3}
\Gamma_1\ll N^{\frac{15\gamma+13}{29}+\varepsilon}\,.
\end{equation}

\textbf{Upper bound for} $\mathbf{\Gamma_2}$
\indent

Using \eqref{Gamma2}, \eqref{fDeltaexpansion} and arguing as in $\Gamma_1$, we deduce
\begin{equation}\label{Gamma2est1}
\Gamma_2\ll\sum\limits_{h=1}^{H}\min\left(\Delta,\frac{1}{h}\right)
\left|\sum\limits_{p\leq N}\big(\psi(-(p+1)^\gamma)-\psi(-p^\gamma)\big)e(\alpha h p^2)\log p \right|+N^{\frac{28\gamma+3}{58}+\varepsilon}\,.
\end{equation}
Denote
\begin{equation}\label{Gu}
G(u)=\sum\limits_{h\leq u}\left|\sum\limits_{p\leq N}\big(\psi(-(p+1)^\gamma)-\psi(-p^\gamma)\big)e(\alpha h p^2)\log p \right|\,.
\end{equation}
By \eqref{Gu} and Abel's summation formula, we get
\begin{align}\label{Abelsum3}
\sum\limits_{h=1}^{H}\min\left(\Delta,\frac{1}{h}\right)
&\left|\sum\limits_{p\leq N}\big(\psi(-(p+1)^\gamma)-\psi(-p^\gamma)\big)e(\alpha h p^2)\log p \right|\nonumber\\
&=\frac{G(H)}{H}+\int\limits_{\Delta^{-1}}^{H}\frac{G(u)}{u^2}\,du\,.
\end{align}
Now \eqref{Nq}, \eqref{HNq}, \eqref{Gamma2est1} -- \eqref{Abelsum3} and Lemma \ref{Sigmaest} lead to
\begin{equation}\label{Gamma2est2}
\Gamma_2\ll(\log H)\max_{\Delta^{-1}\leq u\leq H}\frac{G(u)}{u}+N^{\frac{28\gamma+3}{58}+\varepsilon}\ll N^{\frac{15\gamma+13}{29}+\varepsilon}\,.
\end{equation}
From \eqref{Gammadecomp}, \eqref{Gamma1est3} and \eqref{Gamma2est2}, it follows the statement in the lemma.
\end{proof}

\subsection{The end of the proof}
\indent

Taking into account \eqref{Louerbound}, \eqref{Delta}, \eqref{Gamma} and Lemma \ref{Mainlemma}, we  establish
\begin{equation*}
\sum\limits_{p\leq N\atop{p=[n^{1/\gamma}]}}F_\Delta(\alpha p^2+\beta)\log p\gg N^{\frac{15\gamma+13}{29}+\varepsilon}\,.
\end{equation*}

This completes the proof of Theorem \ref{Theorem}.

\vskip20pt
\footnotesize
\begin{flushleft}
S. I. Dimitrov\\
\quad\\
Faculty of Applied Mathematics and Informatics\\
Technical University of Sofia \\
Blvd. St.Kliment Ohridski 8 \\
Sofia 1756, Bulgaria\\
e-mail: sdimitrov@tu-sofia.bg\\
\end{flushleft}

\begin{flushleft}
Department of Bioinformatics and Mathematical Modelling\\
Institute of Biophysics and Biomedical Engineering\\
Bulgarian Academy of Sciences\\
Acad. G. Bonchev Str. Bl. 105, Sofia 1113, Bulgaria \\
e-mail: xyzstoyan@gmail.com\\
\end{flushleft}

\hskip10pt

\begin{flushleft}
M. D. Lazarova\\
\quad\\
Faculty of Applied Mathematics and Informatics\\
Technical University of Sofia \\
Blvd. St.Kliment Ohridski 8 \\
Sofia 1756, Bulgaria\\
e-mail: meglena.laz@tu-sofia.bg\\
\end{flushleft}

\end{document}